\documentclass[11pt, reqno]{amsart}
\usepackage{fullpage}
\usepackage{bm}

\newif\ifHideFoot
\HideFoottrue  
\HideFootfalse

\newif\ifHideApp
\HideApptrue  
\HideAppfalse  

\ifHideFoot \HideApptrue \fi

\usepackage{amssymb,amsmath,amsthm,amscd,mathrsfs,graphicx, color}
\usepackage[cmtip,all,matrix,arrow,tips,curve]{xy}

\usepackage{hyperref}

\usepackage[usenames,dvipsnames]{xcolor}
\usepackage{xypic, xy}
\hypersetup{colorlinks=true,citecolor=ForestGreen,linkcolor=Maroon,urlcolor=NavyBlue}
\usepackage[normalem]{ulem}
\usepackage{mathtools}
\usepackage{mathpazo}
\usepackage{quiver}
\usepackage{enumitem}

\numberwithin{equation}{section}
\newtheorem{teo}{Theorem}[section]
\newtheorem{pro}[teo]{Proposition}
\newtheorem{lem}[teo]{Lemma}
\newtheorem{cor}[teo]{Corollary}

\newtheorem{fct}[teo]{Fact}

\newtheorem{teoalpha}{Theorem}

\newtheorem{coralpha}[teoalpha]{Corollary}

\newtheorem*{teoalpha*}{Theorem~\ref{T:main}}

\theoremstyle{definition}

\theoremstyle{remark}
\newtheorem{rem}[teo]{Remark}

\ifHideFoot

\newcommand{\Yano}[1]{}
\newcommand{\Jeff}[1]{}
\newcommand{\Charles}[1]{}

\else

\newcommand{\marg}[1]{\normalsize{{
   \color{red}\footnote{{\color{blue}#1}}}{\marginpar[\vskip
   -.25cm{\color{red}\hfill$\implies$\tiny\thefootnote}]{\vskip
    -.2cm{\color{red}$\impliedby$\tiny\thefootnote}}}}}
\newcommand{\Yano}[1]{\marg{(Yano) #1}}
\newcommand{\Jeff}[1]{\marg{(Jeff) #1}}
\newcommand{\Charles}[1]{\marg{(Charles) #1}}

\fi

\DeclareMathOperator{\coniveau}{N}

\global\let\hom\undefined
\DeclareMathOperator{\hom}{hom}

\newcommand{\til}[1]{{\widetilde{#1}}}

\def\et{{\rm \acute et}}
\def\cl{{\rm cl}}
\def\tors{{\mathrm{tors}}}

\def\cx{{\mathbb C}}

\def\rat{{\mathbb Q}}
\def\integ{{\mathbb Z}}

\def\iso{\simeq}

\renewcommand{\bar}[1]{{\overline{#1}}}

\DeclareMathOperator{\aut}{Aut}

\DeclareMathOperator{\spec}{Spec}

\DeclareMathOperator{\A}{A}

\newcommand{\abs}[1]{{\left|#1\right|}}

\newcommand{\invlim}[1]{\lim_{\stackrel{\leftarrow}{#1}}}

\title{The Walker Abel--Jacobi map descends}

\author{Jeffrey D. Achter}
\address{Colorado State University, Department of Mathematics,
 Fort Collins, CO 80523,
 USA}
\email{j.achter@colostate.edu}

\author{Sebastian Casalaina-Martin }
\address{University of Colorado, Department of Mathematics,
 Boulder, CO 80309, USA }
\email{casa@math.colorado.edu}

\author{Charles Vial}
\address{Fakult\"at f\"ur Mathematik, Universit\"at Bielefeld, Germany}
\email{vial@math.uni-bielefeld.de}

\thanks{The first- and second-named authors were partially supported  by grants 637075 and 581058, respectively, from the Simons  Foundation.}

\AtBeginDocument{%
	\def\MR#1{}
}

\begin{document}

 \begin{abstract} 
For a complex projective manifold, Walker has defined a regular homomorphism lifting Griffiths' Abel--Jacobi map on algebraically trivial cycle classes to a complex abelian variety,  which admits a finite homomorphism to the Griffiths intermediate Jacobian.  Recently Suzuki gave an alternate, Hodge-theoretic, construction of this Walker Abel--Jacobi map.  We provide a third construction based on a general lifting property for surjective regular homomorphisms, and prove that the Walker Abel--Jacobi map descends canonically to any field of definition of the complex projective manifold. In addition, we determine the image of the l-adic Bloch map restricted to algebraically trivial cycle classes in terms of the coniveau filtration.
  \end{abstract}

  \maketitle

    Let $H$ be a pure integral Hodge structure of weight-$(-1)$. The \emph{Jacobian} attached to $H$ is the complex torus $$J(H):= \mathrm{F}^0H_\cx \backslash H_\cx \slash H_\tau,$$
    where  $F^\bullet H_\cx$ denotes the Hodge filtration on the complexification $H_\cx:= H\otimes_{\mathbb Z} \cx$ and where, for an abelian group $G$, we denote $G_\tau$ its torsion-free quotient.
   If $X$ is a complex projective manifold, then the cohomology groups $H^{2p-1}(X,\integ(p))$ are naturally endowed with the structure of a pure Hodge structure of weight-$(-1)$. In the seminal paper~\cite{Griffiths}, Griffiths defined an Abel--Jacobi map for homologically trivial cycle classes  $\operatorname{CH}^p(X)_{\hom} := \ker \left( \operatorname{CH}^p(X) \to H^{2p}(X,\integ (p))\right)$\,:
   $$\xymatrix{\operatorname{AJ} : \operatorname{CH}^p(X)_{\hom} \ar[r] &  J^{2p-1}(X):= J\big(H^{2p-1}(X,\integ(p))\big),}$$
    which is in particular functorial with respect to the action of correspondences between complex projective manifolds. 
    Since  algebraically trivial cycles in $\operatorname{CH}^p(X)$ are parametrized by smooth projective complex curves, and since the Abel map
    $\operatorname{CH}^1(C)_0 \to J(C):= J(H^1(C,\integ(1)))$ on degree-0 zero-cycle classes on a curve $C$ is an isomorphism, the image of the Abel--Jacobi map restricted to the subgroup $\operatorname{A}^p(X)\subseteq \operatorname{CH}^p(X)$ of algebraically trivial cycle classes
has image  a subtorus
    $$J_a^{2p-1}(X) \hookrightarrow J^{2p-1}(X)$$ which is algebraic, i.e., an abelian variety, and  called the \emph{algebraic intermediate Jacobian}. The resulting (surjective)  Abel--Jacobi map
   $$\psi^p : \operatorname{A}^p(X) \longrightarrow J_a^{2p-1}(X)$$
   defines a  \emph{regular homomorphism}, meaning that for all pointed  smooth connected complex varieties $(T,t_0)$ and all families of codimension-$p$ cycles $Z \in \operatorname{CH}^p(T\times X)$ the map $T(\cx) \to J_a^{2p-1}(X), t \mapsto \psi^p(Z_t-Z_{t_0})$ is induced by a complex morphism $T \to  J_a^{2p-1}(X)$ of complex varieties.
 The algebraic intermediate Jacobian $J_a^{2p-1}(X)$ can also be described Hodge-theoretically. For $\Lambda$ a commutative ring, consider the \emph{coniveau filtration} $\coniveau^\bullet$\,:
   $$\coniveau^iH^j(X,\Lambda) := \sum \ker \big( H^j(X,\Lambda) \to H^j(X\setminus Z,\Lambda)\big),$$ where the sum runs through all closed subvarieties $Z$ of codimension $\geq i$ in $X$.
   Then the algebraic intermediate Jacobian $J^{2p-1}_a(X)$ is the subtorus of $J^{2p-1}(X)$ corresponding to the inclusion of rational Hodge structures $\coniveau^{p-1}H^{2p-1}(X,\rat(p)) \subseteq H^{2p-1}(X,\rat(p))$\,; precisely,
   $$J^{2p-1}_a(X) = J\big(H^{2p-1}(X,\integ(p))\cap \coniveau^{p-1}H^{2p-1}(X,\cx)\big).$$
 On the other hand, the \emph{Walker intermediate Jacobian} is the complex torus defined as
   $$J^{2p-1}_W(X) := J\big(\coniveau^{p-1}H^{2p-1}(X,\integ(p)) \big).$$
   The inclusion of lattices $\coniveau^{p-1}H^{2p-1}(X,\integ(p)) \subseteq H^{2p-1}(X,\integ(p)) \cap \coniveau^{p-1}H^{2p-1}(X,\cx)$ induces an isogeny of complex tori
   $$\xymatrix{\alpha: J^{2p-1}_W(X) \ar[r]& J^{2p-1}_a(X),}$$
   which in fact is an isogeny of complex abelian varieties, since the pull-back of an ample line bundle on $J^{2p-1}_a(X)$ along the finite map $\alpha$ is ample.
\medskip

Walker has shown that the Abel--Jacobi map on algebraically trivial cycle classes  lifts to the Walker intermediate Jacobian\,:

\begin{teoalpha}[Walker, \cite{Walker}] 
	\label{T:Walker}
	Let $X$ be a  complex projective manifold.
	There exists a regular homomorphism $\psi^p_W$ lifting the Abel--Jacobi map $\psi^p$ along the isogeny $\alpha : J^{2p-1}_W(X)\to J^{2p-1}_a(X)$, i.e., making the following diagram commute\,:
	\[\begin{tikzcd}
	&& {J_W^{2p-1}(X)} \\
	{\operatorname{A}^p(X)} && {J_a^{2p-1}(X).}
	\arrow["{\psi_W^p}", from=2-1, to=1-3]
	\arrow["{\psi^p}"', from=2-1, to=2-3]
	\arrow["{\alpha}", from=1-3, to=2-3]
	\end{tikzcd}\]
\end{teoalpha}

The regular homomorphism $\psi^p_W : \operatorname{A}^p(X) \longrightarrow J_W^{2p-1}(X)$ 
will be called the \emph{Walker Abel--Jacobi map}.
It was first constructed by Walker~\cite{Walker} using Lawson homology\,; recently, Suzuki~\cite{Suzuki-WalkerMap} gave a Hodge-theoretic construction relying solely on Bloch--Ogus theory~\cite{BlochOgus}. That $\psi^p_W$ is regular is \cite[Lem.~7.3]{Walker} or \cite[Cor.~2.6]{Suzuki-WalkerMap}. In addition, it is shown in \cite[Lem.~2.4]{Suzuki-WalkerMap} that $\psi^p_W$ is compatible with the action of correspondences.
 In the case where  $p=1,2,\dim X$,  the usual Abel--Jacobi map $\psi^p$ is universal among regular homomorphisms (see \cite[Thm.~C]{Murre}), and so the Walker Abel--Jacobi map coincides with the usual Abel--Jacobi map (i.e., the isogeny $\alpha$ is an isomorphism), while in general it differs (see Ottem--Suzuki~\cite[Cor.~4.2]{Ottem-Suzuki}) and hence provides a finer invariant for algebraically trivial cycles.

\medskip

The first aim of this paper is to provide a new proof of Walker's Theorem~\ref{T:Walker}\,; see \S \ref{SS:Walker}. Our proof is based on the general lifting Theorem~\ref{T:factorization} for regular homomorphisms (see also Proposition~\ref{P:LiftReg}), 
which we hope could prove useful in other situations, especially in positive characteristic.
 
 \medskip 
As our main new result, we show that if $X$ is defined over a field $K\subseteq \mathbb C$, then the Walker intermediate Jacobian descends to $K$ in such a way that  the diagram of Theorem~\ref{T:Walker}  can be made $\operatorname{Aut}(\cx/K)$-equivariant\,:

\begin{teoalpha}[Distinguished model]\label{T:main} Let $X$ be a smooth projective   variety variety over a field $K\subseteq \mathbb C$.
	Then the isogeny $\alpha: J^{2p-1}_W(X_{\mathbb C}) \to J^{2p-1}_a(X_{\mathbb C})$ of complex abelian varieties descends uniquely to an isogeny $J^{2p-1}_{W,X/K}\to J^{2p-1}_{a,X/K}$ of abelian varieties over  $K$ in such a way that both $\psi^p$ and $\psi^p_W$ are $\operatorname{Aut}(\cx/K)$-equivariant.
\end{teoalpha}

The part of Theorem \ref{T:main} stated for the algebraic intermediate Jacobian was proved in \cite[Thm.~A]{ACMVdmij} (see also~\cite[Thm.~9.1]{ACMVfunctor}).  
We provide two proofs of Theorem~\ref{T:main}. The first one is presented in \S \ref{SS:firstproof}\,; it is based on \cite[Thm.~A]{ACMVdmij}, on the universality of the Walker Abel--Jacobi map among lifts of the Abel--Jacobi map along isogenies (Theorem~\ref{T:Walker-initial}) and on the general descent statement of our lifting Theorem~\ref{T:factorization}. The second one is presented in \S \ref{SS:main} and builds directly upon \cite{ACMVdmij}.
We note also here that, as in \cite[Thm.~A]{ACMVdmij} and \cite[Prop.~3.1]{ACMVnormal}, which concern the case of the algebraic intermediate Jacobian, 
the $K$-structure in Theorem \ref{T:main} for the Walker intermediate Jacobian and Walker Abel--Jacobi map is stable under field extensions $K\subseteq L\subseteq \mathbb C$ (Remark~\ref{R:L/K}), 
and independent of the embedding of $K$ into $\mathbb C$ (Remark \ref{R:EmbKinCC}).  
As a consequence, the kernel of the Walker Abel--Jacobi map is independent of the choice of embedding of $K$ into $\mathbb C$\,; the analogous statement for the Abel--Jacobi map on algebraically trivial cycle classes is \cite[Rem.~3.4]{ACMVnormal}.
\medskip 

From our second approach to proving Theorem~\ref{T:main} we obtain two applications.\medskip

First, 
we obtain the following proposition, which  provides further arithmetic significance to the Walker Abel--Jacobi map, by showing that the torsion-free quotient of  $\coniveau^{p-1}H_{\operatorname{\acute{e}t}}^{2p-1}(X_{\mathbb C},\integ_\ell(p))$ can be modeled by an abelian variety independently of $\ell$\,:

\begin{coralpha}[Modeling coniveau integrally]
	\label{C:model}
   Let $X$ be a smooth projective variety over a field $K\subseteq \cx$. Then for all integers $p$, 
  the model $J^{2p-1}_{W,X/K}$ over $K$ of the Walker intermediate Jacobian $J^{2p-1}_W(X_{\mathbb C})$ (Theorem~\ref{T:main}) has the property that for all primes $\ell$ we have canonical isomorphisms of $\operatorname{Aut}(\cx/K)$-representations
   $$T_\ell J^{2p-1}_{W,X/K}  \simeq \coniveau^{p-1}H_{\operatorname{\acute{e}t}}^{2p-1}(X_{\mathbb C},\integ_\ell(p))_\tau.$$
\end{coralpha}

This result is established in \S \ref{SS:Mazur}. It was established with $\rat_\ell$-coefficients in \cite[Thm.~A]{ACMVdmij} with the model of the algebraic intermediate Jacobian over $K$ in place of that of the Walker intermediate Jacobian.   We direct the reader to \cite{ACMVbloch} for more details, and in particular,  the connection to a question of Mazur~\cite{mazurprobICCM}.
\medskip

Second, for any smooth projective variety $X$ over an algebraically closed field and for any prime $\ell$ invertible in $X$, Bloch~\cite{bloch} has defined a map 
$\lambda^p : \operatorname{CH}^p(X)[\ell^\infty] \to H^{2p-1}_{\mathrm{\acute et}}(X,\mathbb Q_\ell/\integ_\ell(p))$ on $\ell$-primary torsion. 
The \emph{$\ell$-adic Bloch map} $T_\ell\lambda^p : T_\ell\operatorname{CH}^p(X) \to H^{2p-1}_{\mathrm{\acute et}}(X,\integ_\ell(p))_\tau$ is then obtained by taking Tate modules and making the identification $T_\ell H^{i}_{\mathrm{\acute et}}(X,\mathbb Q_\ell/\integ_\ell(j)) = H^{i}_{\mathrm{\acute et}}(X,\mathbb \integ_\ell(j))_\tau$\,; we refer to \cite[(2.6.5)]{suwa}, and to \cite[\S A.3.3]{ACMVbloch}, for more details. Here, the Tate module associated to an $\ell$-primary torsion abelian group $M$ is the group $T_\ell M:= \varprojlim M[\ell^n]$.  Thanks to our approach to lifting regular homomorphisms along isogenies, together with the existence of the Walker Abel--Jacobi map, we determine the image of $T_\ell \lambda^p$ restricted to algebraically trivial cycle classes\,:

\begin{coralpha}\label{C:imageBloch}
Let $X$ be a smooth projective variety over a field $K$ of characteristic zero. Then 
$$\operatorname{im} \big(T_\ell\lambda^p : T_\ell \operatorname{A}^p(X_{\bar K}) \longrightarrow H_{\operatorname{\acute{e}t}}^{2p-1}(X_{\bar K},\integ_\ell(p))_\tau \big) = \coniveau^{p-1}H_{\operatorname{\acute{e}t}}^{2p-1}(X_{\bar K},\integ_\ell(p))_\tau$$
for all primes $\ell$.
\end{coralpha}
This extends  \cite[Prop.~5.2]{suwa} (see also \cite[Prop.~2.1]{ACMVbloch}), where the images of the usual Bloch map $\lambda^p$ and of $T_\ell \lambda^p \otimes \rat_\ell$, both restricted to algebraically trivial cycle classes, were determined.

\subsection*{Acknowledgments} We thank Fumiaki Suzuki for mentioning to us that the Walker Abel--Jacobi map does not lift along non-trivial isogenies (Theorem~\ref{T:Walker-initial}), and for bringing \cite{BF84} to our attention.  We also thank the referee for helpful suggestions.

\section{Lifting regular homomorphisms along isogenies}

\subsection{An elementary fact}
We start with the following elementary fact, which will be used recurringly throughout this note.
\begin{fct}\label{F:el}
	Let $f: D \to G$ and $\alpha : G' \to G$ be homomorphisms of abelian groups. Assume $D$ is divisible and that $\ker \alpha$ is finite. Then there exists at most one homomorphism $f' : D \to G'$ such that $\alpha \circ f' = f$, i.e., such that the following diagram commutes\,:
	\[\begin{tikzcd}
	&& {G'} \\
	{D} && {G.}
	\arrow["{f'}", from=2-1, to=1-3]
	\arrow["{f}"', from=2-1, to=2-3]
	\arrow["{\alpha}", from=1-3, to=2-3]
	\end{tikzcd}\]
Moreover, if there is a group $H$ acting on $D$, $G$, and $G'$, and $f$ and $\alpha$ are $H$-equivariant, then $f'$, if it exists,  is $H$-equivariant, as well. \qed
\end{fct}

As a first consequence, note that since for a smooth complex  projective variety   $X$ one has that $\operatorname{A}^p(X)$ is a divisible group (e.g., \cite[Lem.~7.10]{BlochOgus}), there is at most one homomorphism $\psi^p_W : \operatorname{A}^p(X) \rightarrow J_W^{2p-1}(X)$  such that $\alpha \circ \psi^p_W = \psi^p$\,; i.e., there is at most one lifting of the Abel--Jacobi map to the Walker intermediate Jacobian.

\subsection{Lifting homomorphisms of abelian varieties along isogenies}
We have the following elementary  lemma on lifting morphisms of
abelian varieties.  (Recall that for an abelian variety $A$ over a
field $K$ of positive characteristic, the $N$-torsion group scheme
$A[N]$ may carry strictly more information than the abstract group of
points $A[N](\bar K)$.)

\begin{lem}
  \label{L:LiftAb}
  Let $f:B \to A$ be a morphism of abelian varieties over a field $K$, and let $\alpha:A' \to A$ be an isogeny of abelian varieties over $K$.  The following are equivalent\,:
  \begin{enumerate}
\item There exists a lift $ f':B\to A'$ of $f$\,; i.e.,  there is a commutative diagram
$$
\xymatrix{
&&A'\ar[d]^\alpha\\
B \ar[rr]_f \ar[rru]^{f'}&& A.\\
}
$$

\item There exists a lift of $f$ restricted to
  torsion schemes\,; i.e., for each natural number $N$ there is a commutative diagram of finite group schemes
\begin{equation}
  \label{E:diagLiftTors}
\xymatrix{
	&&A'[N]\ar[d]^{\alpha[N]}\\
	B[N] \ar[rr]_{f[N]} \ar[rru]^{ (f[N])'}&& A[N]\\
}
\end{equation}
such that $(f[MN])'|_{B[N]} = (f[N])'$.
  \end{enumerate}
If $\alpha$ is separable (equivalently, \'etale), and $\Omega/K$ is any field extension with $\Omega$ algebraically closed,
 then (1) and (2) are also equivalent to each of the following conditions\,:

\begin{enumerate}[resume]
\item  There exists a group-theoretic lift $(f(\Omega)_{\mathrm{tors}})'
  :B(\Omega)_{\mathrm{tors}} \to A'(\Omega)_{\mathrm{tors}}$ of $f(\Omega)$ restricted to
  torsion points\,; i.e.,  there is a commutative diagram of torsion
  abelian groups
$$
\xymatrix{
	&&A'(\Omega)_{\mathrm{tors}}\ar[d]^{\alpha(\Omega)_{\mathrm{tors}}}\\
	B(\Omega)_{\mathrm{tors}} \ar[rr]_{f(\Omega)_{\mathrm{tors}}} \ar[rru]^{
          (f(\Omega)_{\mathrm{tors}})'}&& A(\Omega)_{\mathrm{tors}}.\\
}
$$

\item   For all prime numbers $l$ there exists a group-theoretic lift $(T_lf)'  :T_l B \to T_l A'$ of $T_l f$, the map on Tate modules\,; i.e.,  there is a commutative diagram
$$
\xymatrix{
&&T_l A' \ar[d]^{T_l \alpha }\\
T_l B  \ar[rr]_{T_l f} \ar[rru]^{(T_lf)' }&& T_lA.\\
}
$$
\item For all prime numbers $l$, we have $\operatorname{im} (T_lf) \subseteq \operatorname{im}(T_l\alpha)$.
\end{enumerate}
Finally, if any of the lifts in (1)--(4) exist, they are unique.  In particular, $(f')_{\mathrm{tors}}= (f_{\mathrm{tors}})'$, $T_l(f') = (T_lf)'$ and, for any extension field $L/K$, $(f(L)_{{\mathrm{tors}}})' = f'(L)_{{\mathrm{tors}}}$.
\end{lem}

\begin{proof}
  The uniqueness of the lift $f'$ follows  from Fact
  \ref{F:el}\,; and (1) clearly implies (2).  Moreover, (2) implies
  (3), and (3) implies (2) over an algebraically closed field of
  characteristic zero.  Conditions (3) and (4) are obviously equivalent; (4) and (5) are equivalent because each $T_l\alpha$ is an inclusion. 
  
To show (2) implies (1), suppose there exists a suitable lift of $f$
on torsion schemes.  By rigidity of homomorphisms of abelian varieties, we may assume that $K$ is perfect.  Using the uniqueness of $f'$ and Galois descent, we may and do assume $K$ is algebraically closed.

We start by reducing to the case where $f$ is an isogeny.  To this end, consider the diagram 
$$
\xymatrix{
&&A'' \ar@{^(->}[r]^{\iota '} \ar@{->>}[d]^{\alpha '}& A' \ar@{->>}[d]^\alpha \\
B\ar@{->>}[r]^{f_{\text{conn}}}& B' \ar@{->>}[r]^{f_{\text{fin}}}& B'' \ar@{^(->}[r]^{\iota} & A\\
}
$$
where  $B'':=\operatorname{im}(f)\subseteq A$, the morphisms
$f:B\stackrel{f_{\text{conn}}}{\to}B'\stackrel{f_{\text{fin}}}{\to} B''$ give the Stein
factorization, $\iota$ is the natural inclusion, and $A''=B''\times_A
A'$.  Explicitly, $B' = B/((\ker f)^0_{\mathrm{red}})$ is the quotient
of $B$ by the largest sub-abelian variety contained in $\ker(f)$.

Fix a prime $l$ and consider $l$-primary torsion.  Using the lift $(f[l^\infty])'$, we have the diagram:
$$
\xymatrix{
&&A''[l^\infty] \ar@{^(->}[r]^{\iota ' [l^\infty]} \ar@{->>}[d]^{\alpha '[l^\infty]}& A' [l^\infty]\ar@{->>}[d]^{\alpha[l^\infty]} \\
B[l^\infty]\ar@{->>}[r]^{f_{\text{conn}}[l^\infty]} \ar@/^1pc/[rru]^{(f[l^\infty])'}& B'[l^\infty] \ar@{->>}[r]^{f_{\text{fin}}[l^\infty]} \ar@/^1pc/[l]& B'' [l^\infty]\ar@{^(->}[r]^{\iota[l^\infty]} & A[l^\infty].\\
}
$$
 The splitting of the map $f_{\text{conn}}[l^\infty]$ is elementary, 
 since whenever one has a short exact sequence of abelian varieties 
 the induced maps  on $l$-primary torsion give a split exact
 sequence   (taking $l$-primary torsion is exact since the kernel
 is divisible, and then free modules are projective).   (If $l =
 \operatorname{char}(K)$, an appeal to Dieudonn\'e modules gives the
 same conclusion.)

 Thus, we now assume $f$ and $\alpha$ are isogenies.  Suppose briefly
 that 
 $\operatorname{char}(K)=0$\,; then $f$ and $\alpha$ are \'etale.  The cover $f$ factors through $\alpha$ if and only if the induced map on \'etale fundamental groups $f_*:\pi_1^\et(B, 0_B) \to \pi_1^\et(A,0_A)$ factors through $\alpha_*:\pi_1^\et(A',0_A) \to \pi_1^\et(A,0_A)$.  For an abelian variety $D/K$, there is a canonical isomorphism $\pi_1^\et(D,0_D) \iso \invlim N D[N](K)$. By taking the inverse limit of the maps of finite groups $(f[N])'(K)$, we see that the condition on fundamental groups is equivalent to (2).

Now suppose instead that $K$ is algebraically closed of
 positive characteristic.
Then $f$, while possibly not \'etale, is at
 least a torsor over $X$ under the finite commutative group scheme
 $\ker(f)$.  Consequently, it is classified by a quotient of Nori's
 fundamental group  scheme $\pi_1^{\mathrm{Nori}}(A)$ \cite{nori76}.  Moreover, for an abelian variety $D/K$, we
 have $\pi_1^{\mathrm{Nori}}(D) = \invlim N D[N]$  \cite{nori83}.
 Consequently, condition (2) is again equivalent to the hypothesis
 that the cover $f$ factors through $\alpha$.

 Finally, suppose $\alpha$ is \'etale by hypothesis and that (3)
 holds.  As noted above, it suffices to consider the case where $K$ is
 algebraically closed of positive characteristic and $f:B \to A$ is an
 isogeny.  Now, any isogeny $g:D \to C$ of abelian varieties over $K$
 admits a canonical factorization $g = g_{\et}\circ g_{\mathrm i}$,
 where $g_{\mathrm i}: D \to D_{\mathrm i} := D/(\ker g)^0$ is purely
 inseparable and $g_{\et}$ is \'etale. Since $\alpha$ is
 \'etale, $f$ factors through $\alpha$ if and only if
 $f_{\et}:B_{\mathrm i} \to A$ factors through $\alpha$.  Because $f_{\mathrm i}$ is a
 universal homeomorphism, the map of groups $\til
 f_{\mathrm{tors}}(K)$ descends to a map of groups $\til f_{\et,
   \mathrm{tors}}(K): B_{i,\mathrm{tors}}(K) \to
 A'_{\mathrm{tors}}(K)$. Now $f_{\et}$ and $\alpha$ are \'etale
 isogenies and we may argue using fundamental groups as before, while
 recalling that (in all characteristics) $\pi_1^\et(D,0_D) \iso \invlim N D[N](K)$.  The same argument, combined with the canonical isomorphism $\pi_1^\et(D,0_D) \iso \prod_l T_l D$, shows that (4) implies (1), as well.
 \end{proof}

\subsection{Lifting regular homomorphisms along isogenies}
From Lemma~\ref{L:LiftAb} we get the following lifting criterion for regular homomorphisms\,:
\begin{pro}\label{P:LiftReg}
  Let $K$ be a field,
   and $\Omega/K$ an algebraically closed
  extension.  Let $X/K$ be a smooth projective variety, let $A/K$ be
  an abelian variety over $K$, let $\phi:\operatorname{A}^p(X_\Omega)\to A(\Omega)$ be an $\operatorname{Aut}(\Omega/K)$-equivariant regular homomorphism, and let $\alpha:A'\to A$ be
  an \'etale isogeny of abelian varieties over $K$.   
Then the following are equivalent\,:
 \begin{enumerate}

 \item  The  $\operatorname{Aut}(\Omega/K)$-equivariant  regular homomorphism $\phi$ 
 lifts to $A'$, in the sense that there is a commutative diagram of 
  $\operatorname{Aut}(\Omega/K)$-equivariant regular homomorphisms 
\begin{equation*}
\xymatrix{
& & A'(\Omega) \ar[d]^{\alpha(\Omega)}\\
\operatorname{A}^p(X_\Omega)\ar[rr]_\phi \ar@{-->}[rru]^{\phi'}&& A(\Omega).
}
\end{equation*}

\item The homomorphism $\phi$ lifts on torsion, in the sense that
  there is a commutative diagram of torsion abelian groups 
\begin{equation*}
\xymatrix{
&& A'(\Omega)_\tors \ar[d]^{ \alpha(\Omega)_\tors}\\
 \operatorname{A}^p(X_\Omega)_\tors \ar[rr]_<>(0.5){ \phi_\tors} \ar@{-->}[rru]^{(\phi_\tors)'}&& A(\Omega)_\tors.
}
\end{equation*}

 \item  For all prime numbers $l$ there exists a group-theoretic lift
  $(T_l\phi  )'  :T_l \operatorname{A}^p(X_{\Omega}) \to T_l A'$ of $T_l \phi$, the map on Tate modules\,; i.e.,  there is a commutative diagram
\begin{equation*}
\xymatrix{
&& T_lA' \ar[d]^{ T_l\alpha}\\
 T_l\operatorname{A}^p(X_\Omega) \ar[rr]_<>(0.5){ T_l\phi} \ar@{-->}[rru]^{(T_l\phi)'}&& T_lA.
}
\end{equation*}

\item For all prime numbers $l$, we have $\operatorname{im}(T_l\phi)\subseteq \operatorname{im}(T_l\alpha)$.

\end{enumerate}
Finally, if any of the lifts in (1)--(3) exist,
 then they are unique and $\aut(\Omega/K)$-equivariant.  In particular, $(\phi')_{\mathrm{tors}}= (\phi_{\mathrm{tors}})'$ and $(T_l\phi)' = T_l(\phi')$.
\end{pro}

\begin{proof} The uniqueness and $\operatorname{Aut}(\Omega/K)$-equivariance of the lifts follows  from Fact \ref{F:el}.  
  Clearly (1) implies (2), and (2) implies (3) by taking Tate modules.  The equivalence of (3) and (4) is obvious since 
  $T_l\alpha$ is injective.
Thus we will show (3) implies (1).  

 Let $(T,t_0)$ be a smooth pointed variety over $\Omega$, and let $\Gamma\in \operatorname{CH}^p(T\times_{\Omega} X_{\Omega})$.  Then we have a commutative diagram $$
\xymatrix@C=3em{
T\ar[r]^<>(0.5){t\mapsto t-t_0} \ar[rd]_{w_\Gamma;\  t\mapsto \Gamma_t-\Gamma_{t_0}}& \operatorname{A}_0(T) \ar[r]^{\operatorname{alb}\quad } \ar[d]_{\Gamma_*}& \operatorname{Alb}(T)(\Omega)\ar[d]^f & A'(\Omega) \ar[ld]^{\alpha}\\
& \operatorname{A}^p(X_\Omega)\ar[r]^\phi& A(\Omega)\\
}
$$
where the top row is the pointed Albanese, and the right vertical arrow $f$ comes from the universal property of algebraic representatives, together with the facts that Albaneses are algebraic representatives, and that $\phi\circ \Gamma_*$ can easily be confirmed to be a regular homomorphism.  

On Tate modules we obtain a diagram
\begin{equation}\label{E:diagTors}
\xymatrix@C=3em@R=4em{
 & T_l\operatorname{A}_0(T) \ar[r]_<>(0.5)\sim^<>(0.5){T_l\operatorname{alb}}
  \ar[d]_{\Gamma_*}& T_l\operatorname{Alb}(T)\ar[d]^{T_lf} & T_lA'\ar[ld]^{T_l\alpha}\\
& T_l\operatorname{A}^p(X_\Omega)\ar[r]^<>(0.5){T_l\phi}
  \ar@{-->}@/^.5pc/[rru]^<>(0.3){(T_l\phi)'}& T_lA\\
}
\end{equation}
where the lift $(T_l\phi)'$ is provided by assumption (3).
The isomorphism on Tate modules  for the Albanese map  comes from Roitman's theorem (see, e.g., \cite[Prop.~A.29]{ACMVbloch}).  
Since we assume $\alpha$ is \'etale, by Lemma \ref{L:LiftAb} we obtain a lift $f':\operatorname{Alb}(T)\to A'$ of $f$ giving a commutative diagram
 \begin{equation}\label{E:diag}
 	\xymatrix@C=3em{
 		T(\Omega)\ar[r]^<>(0.5){t\mapsto [t]-[t_0]} \ar[rd]_{w_\Gamma;\  t\mapsto \Gamma_t-\Gamma_{t_0}}& \operatorname{A}_0(T) \ar[r]^{\operatorname{alb}\quad } \ar[d]_{\Gamma_*}& \operatorname{Alb}(T)(\Omega)\ar[d]^{f(\Omega)}  \ar[r]^{\quad f'(\Omega)}& A'(\Omega) \ar[ld]^{\alpha(\Omega)}\\
 		& \operatorname{A}^p(X_\Omega)\ar[r]^\phi& A(\Omega)\\
 	}
 \end{equation}
It follows immediately that if $\phi$ lifts to an abstract homomorphism $ \phi':\operatorname{A}^p(X_\Omega)\to A'(\Omega)$, then $\phi'$ is a regular homomorphism.  
Thus we have reduced the problem to showing that $\phi$ lifts as an abstract homomorphism to a homomorphism $\phi':\operatorname{A}^p(X_\Omega)\to A'(\Omega)$.  
\medskip 

Over an algebraically closed field, algebraically trivial cycles are parameterized by smooth projective curves~\cite[Ex.~10.3.2]{fulton}.
In other words, $ \operatorname{A}^p(X_\Omega)$ is covered by the images of $\Gamma_*: \operatorname{A}_0(T) \to \operatorname{A}^p(X_\Omega)$,
 where $T$ runs through pointed smooth projective curves over $\Omega$ and $\Gamma$ over correspondences in $\operatorname{CH}^p(T\times_\Omega X_\Omega)$.  
 Now since $\operatorname{A}_0(T)$ is divisible, it follows that $\Gamma_*(\operatorname{A}_0(T))$ is divisible\,; 
 therefore, by the uniqueness of lifts  (Fact~\ref{F:el}) it is enough to show that ${f}' \circ \operatorname{alb}$ in~\eqref{E:diag} factors through $\Gamma_*(\operatorname{A}_0(T))$ in the case where $T$ is a smooth projective curve. 
 In other words, taking $T$ to be a smooth projective curve over $\Omega$, and given any $\gamma\in \operatorname{A}_0(T)$ such that $\Gamma_*(\gamma)=0$, 
 we must show that $({f}' \circ \operatorname{alb})(\gamma)=0$.   \medskip

The first observation is that this is clear if $\Omega$ is the algebraic closure of a finite field. 
 Indeed, in that case $\operatorname{A}_0(T)$ is a torsion group, since the Albanese map $\operatorname{A}_0(T) \to \operatorname{Alb}_T(\Omega)$ is an isomorphism and  closed points of an abelian variety over a finite field are torsion.
Thus  $\gamma$ is torsion.  
  Decomposing torsion in  $\operatorname{A}_0(T)$
    into a direct sum of $l$-power torsion, we can work one prime at a time.  Now we make the following elementary observation\,: given any homomorphism of groups $h:D\to G$ where $D$ is divisible, and any $x\in D[l^\infty]$, we have that $h(x)=0$ if for some lift $x_l$ of $x$ to  $T_lD$ (which exists since $D$ is divisible), we have that $(T_lh)(x_l)=0$.   
  Consequently, taking Tate modules in \eqref{E:diag} and using the lift $(T_l\phi)'$ \eqref{E:diagTors}, we see that $\operatorname{alb}(\gamma)=0$.  
\medskip

We now deduce the general case from the case of finite fields, \emph{via} a specialization argument.  For this we use the terminology of  regular homomorphisms from \cite{ACMVfunctor}, which is much better suited to the relative setting. Since
all objects considered here are of finite type, the data $X$, $T$, $\Gamma$,
$A$, $A'$, $\alpha$ and $\gamma$ descend to a field $L$ which is
finitely generated over the prime field.  A standard spreading
argument produces a smooth ring $R$, finitely generated as a
$\integ$-algebra and with fraction field $L$, and smooth $\bm
X$,
$\bm T$, $\bm A$, $\bm A'$ over $S = \spec(R)$, as well as $\bm \gamma
\in \mathscr A^1_{\bm T/S}(S)$,
 whose generic fibers are the corresponding
original data.  Let $\abs{S}^\cl$ be the set of points of $S$ with
finite residue fields\,; then $\abs{S}^\cl$ is topologically dense in~$S$.

From \cite{ACMVfunctor}, there exists a diagram
 \begin{equation}\label{E:diagS}
 	\xymatrix@C=3em{
 		 \mathscr {A}^1_{\bm T/S} \ar[r]^{\operatorname{alb}} \ar[d]_{\Gamma_{S*}}& \operatorname{Alb}_{\bm T/S} \ar[d]^{f_S}  \ar[r]^{f'_S}& \bm A' \ar[ld]^{\alpha_S}\\
 		\mathscr {A}^p_{\bm X/S}\ar[r]^\Phi& \bm A \\
 	}
 \end{equation}
where $\Phi:\mathscr A^p_{\bm X/S}\to \bm A$ is a regular homomorphism,
the Albanese homomorphism is the universal regular homomorphism for $0$-cycles \cite[Lem.~7.5]{ACMVfunctor} 
and the remaining morphisms are
extensions of those in \eqref{E:diag}. Set $\bm a' = (f'_S \circ
\operatorname{alb})(\bm \gamma) \in \bm A'(S)$.

Now suppose $s \in \abs{S}^\cl$.  Then pullback of \eqref{E:diagS} yields a
diagram of objects over $s = \spec(\kappa(s))$, where specialization
of cycles is provided by \cite[20.3.5]{fulton}.  We have seen that for
each such $s$, $\bm a'_s =0 \in \bm A'(s)$.  Using the density of
$\abs{S}^\cl$, we see that $\bm a'=0$, and in particular its generic fiber
$(f'\circ \operatorname{alb})(\gamma)$ is zero.
\end{proof}

\begin{rem}[Regular homomorphisms and mini-versal cycle classes]
	\label{R:miniverse}
	Given a surjective $\operatorname{Aut}(\Omega/K)$-equivariant regular homomorphism $\phi:\operatorname{A}^p(X_\Omega)\to A(\Omega)$, there is a cycle class $\Gamma\in \operatorname{CH}^p(A\times _K X)$ (which we call a \emph{mini-versal cycle class}) 
	such that the associated map $\psi_\Gamma:A\to A$, induced on $\Omega$-points by $a\mapsto \Gamma_a-\Gamma_0\mapsto \phi(\Gamma_a-\Gamma_0)$,  is given by multiplication by some non-zero integer~$r$  \cite[Lem.~4.7]{ACMVfunctor}.  
	One can immediately see from the definition that given any \'etale isogeny $\alpha :A'\to A$ through which $\phi$ factors, one has $(\deg \alpha) \mid (\deg r\cdot \operatorname{Id}_A)= r^{2\dim A}$.   In particular, if there is a universal cycle class (i.e., $r=1$), then $\phi$ does not factor through any non-trivial isogeny~$A'\to A$.  
\end{rem}

\medskip

We obtain the following consequence of Proposition \ref{P:LiftReg}, establishing the existence of a universal lifting of a surjective regular homomorphism along isogenies. Together with Corollary~\ref{C:factorization}, this extends 
\cite[Thm.~0.1]{BF84} to the case of arbitrary fields.  Note also that the proof of \cite[Thm.~0.1]{BF84} is incorrect. (On the bottom of \cite[p.362]{BF84},  it is assumed that the map $u : B(k) \to A^q(X)$ is a homomorphism, so that the image of $u$ is a subgroup of $A^q(X)$. There, $X$ is a smooth projective variety over an algebraically closed field $k$, $B$ is an abelian variety over $k$, and $u : b \mapsto Z_*([u]-[0])$ is the map induced by a cycle  $Z\in \operatorname{CH}^p(B\times_k X)$. However, this is not the case in general. Indeed, consider the special instance where $X=B$ is an abelian variety of dimension $>1$ over an uncountable algebraically closed field $k$ and where $Z = \Delta_B$ is the diagonal cycle class. Then the map $u : B(k)  \to \operatorname{A}_0(B), b \mapsto [b]-[0]$ is not a homomorphism since by \cite[Thm.~3.1(a)]{Bloch76} there exist $b_1$ and $b_2$ in $B(k)$ such that $[b_1+b_2] \neq [b_1] + [b_2] - [0]$ in $\operatorname{A}_0(B)$\,; see also \cite[p.309]{Murre88}.)

\begin{teo}[Universal lift of surjective regular homomorphisms along \'etale isogenies]\label{T:factorization}
	Let $K$ be a  field, and $\Omega/K$ an algebraically closed
	extension.  Let $X/K$ be a smooth projective variety, let $A/K$ be
	an abelian variety over $K$, and let $\phi:\operatorname{A}^p(X_\Omega)\to A(\Omega)$ be a surjective regular homomorphism.
	Then
	there exist an \'etale isogeny $\alpha : \widetilde A \to A_\Omega$, characterized by the condition  $ \operatorname{im}(T_l\alpha)= \operatorname{im}(T_l\phi)$ for all primes $l$,  and a surjective regular homomorphism $\tilde \phi :  \operatorname{A}^p(X_\Omega)\to \widetilde A(\Omega)$ which is initial among all regular lifts of $\phi$ along \'etale isogenies $A' \to A_\Omega$.  
	
	Moreover, if $\phi$ is $\operatorname{Aut}(\Omega/K)$-equivariant, then $\widetilde A$ admits a unique model over $K$ such that $\tilde \phi$ is $\operatorname{Aut}(\Omega/K)$-equivariant, and the isogeny $\alpha$ descends to $K$.
\end{teo}

\begin{proof}
	Using a mini-versal cycle class as in Remark~\ref{R:miniverse}, 
	one sees that $(\prod T_l
	\phi)(\A^p(X_\Omega))$ has finite index in $\prod T_l A \iso
	\pi_1^\et(A_\Omega,0)$.  Consequently,  
	it determines an \'etale
	isogeny $\widetilde A \to A_\Omega$ over $\Omega$\,; by Proposition~\ref{P:LiftReg}(4), there is a surjective regular homomorphism $\tilde \phi: \A^p(X) \to \widetilde A(\Omega)$ which lifts $\phi$ and which is initial among all regular lifts of $\phi$ along \'etale isogenies $A' \to A_\Omega$ over $\Omega$.

	Suppose now that $\phi$ is $\operatorname{Aut}(\Omega/K)$-equivariant, and briefly assume $K$ perfect. The unicity of the model over $K$ follows from the elementary Fact~\ref{F:el}. 
	Its existence follows from the universality of~$\tilde \phi$\,: for all $\sigma \in \operatorname{Aut}(\Omega/K)$, one obtains an isomorphism $g_\sigma: \widetilde A \to \widetilde{A}^\sigma$ over $\Omega$, where $\widetilde{A}^\sigma$ is the pull-back of $\widetilde{A}$ along $\sigma : \Omega \to \Omega$, making the following diagram commute
	\[\begin{tikzcd}
	&&& {\widetilde{A}(\Omega)} \\
	{\operatorname{A}^p(X_\Omega)} && {A_\Omega(\Omega)} \\
	&&& {\widetilde{A}^\sigma(\Omega)}
	\arrow["\phi", from=2-1, to=2-3]
	\arrow["{\tilde{\phi}^\sigma}"', from=2-1, to=3-4]
	\arrow["{\tilde \phi}", from=2-1, to=1-4]
	\arrow["{g_\sigma(\Omega)}", from=1-4, to=3-4]
	\arrow["\alpha", from=1-4, to=2-3]
	\arrow["{\alpha^\sigma}"', from=3-4, to=2-3]
	\end{tikzcd}\]
	Here $\tilde{\phi}^\sigma$ and $\alpha^\sigma$ are obtained from the action of $\sigma$ on $\operatorname{A}^p(X_\Omega)$ and on $A_\Omega$, and from the canonical $\sigma$-morphism $\widetilde{A}^\sigma \to \widetilde{A}$.
	To conclude, one checks as in the proof of \cite[Thm.~4.4]{ACMVdcg} that the isomorphisms $g_\sigma^{-1}$ for $\sigma \in \operatorname{Aut}(\Omega/K)$ define a Galois-descent datum
	on the isogeny $\alpha : \widetilde{A} \to A_\Omega$.

If $K$ is a non-perfect field, let $K^{\mathrm{perf}}$ be the perfect closure of $K$ inside $\Omega$.  
From what we have seen, since $\operatorname{Aut}(\Omega/K^{\mathrm{perf}})  \subseteq
\operatorname{Aut}(\Omega/K)$,  $\widetilde A$ descends to $K^{\mathrm{perf}}$. Because in fact 
$\operatorname{Aut}(\Omega/K^{\mathrm{perf}})  =
\operatorname{Aut}(\Omega/K)$, it suffices to show that $\alpha: \widetilde A
\to A_{K^{\mathrm{perf}}}$ descends to $K$. Now, by definition, the homomorphism $\alpha$ factors
through the $K^{\mathrm{perf}}/K$-image $\widetilde A\to
\operatorname{im}_{K^{\mathrm{perf}}/K}(\widetilde A)_{K^{\mathrm{perf}}}$, which exists due to~\cite[Thm.~4.3]{conradtrace}. Since $\alpha:\widetilde A \to A_{K^{\mathrm{perf}}}$ is
\'etale and $K^{\mathrm{perf}}/K$ is primary, the canonical map $\widetilde A\to
\operatorname{im}_{K^{\mathrm{perf}}/K}(\widetilde A)_{K^{\mathrm{perf}}}$, which always has
connected kernel \cite[Thm.~4.5(3)]{conradtrace}, is an isomorphism,
and $\widetilde A$ and $\alpha$ descend
canonically to $K$.
\end{proof}

We derive the following characterization of surjective regular homomorphisms that do not lift along non-trivial isogenies in terms of their kernels\,:
	
	\begin{cor} \label{C:factorization}
Let $X$ be a smooth projective variety over an algebraically closed field $\Omega$ and let $\phi: \operatorname{A}^p(X)\to A(\Omega)$ be a surjective regular homomorphism. 
Then the following statements are equivalent\,:
\begin{enumerate}
	\item $\ker \phi$ is divisible.
	
	\item $\ker \phi_{\operatorname{tors}}$ is divisible.
	
	\item $T_l \phi$ is surjective for all primes $l$.
	
	\item $\phi$ does not factor through any non-trivial \'etale  isogeny $\alpha:A'\to A$. 
\end{enumerate}
	\end{cor}
	\begin{proof}
	The argument in the proof of Theorem~\ref{T:factorization} says that (3) and (4) are equivalent (recall from Proposition~\ref{P:LiftReg} that a group-theoretic lift of a regular homomorphism along an isogeny is a regular homomorphism).
	The elementary commutative algebra  Lemma \ref{L:CAlim} below gives the equivalence of (1) and (3). Finally, since surjective regular homomorphisms are surjective on torsion (see \cite[Rem.~3.3]{ACMVdmij}), Lemma \ref{L:CAlim} below also gives that  $T_l\phi$ being surjective for all $l$ is equivalent to $\ker (\phi_{\operatorname{tors}})$ being $l$-divisible for all primes $l$, i.e., that (2) is equivalent to (3).
\end{proof}

\begin{lem} \label{L:CAlim}
	Suppose that we have a short exact sequence of abelian groups
	$$
	0\to H \to D\to G\to 0
	$$
	with $D$ an $l$-divisible group.  Then the left exact sequence 
	$$
	0\to T_l H\to T_lD\to T_lG
	$$
	is right exact if and only if $H$ is $l$-divisible.  
	
	If in addition $D_{\operatorname{tors}}\to G_{\operatorname{tors}}$ is surjective, then this is also equivalent to $H_{\operatorname{tors}}$ being $l$-divisible.
\end{lem}
\begin{proof} 
	Since $D$ is $l$-divisible, we have for all $n>0$ exact sequences
	$$
	0\to H[l^n]\to D[l^n]\to G[l^n] \to H/l^nH \to 0.
	$$
	Using that $A/l^nA=0$ and $\varprojlim^1_n A[l^n]= 0$ for any $l$-divisible abelian group $A$, we  obtain by passing to the inverse limit a short exact sequence
	$$
	0\to T_lH \to T_lD \to T_lG\to 0.
	$$
	
	Conversely,  if $H$ is not $l$-divisible, let us assume that $H/l^nH\ne 0$ for all $n\ge n_0$.  
	In particular $D[l^n]\to G[l^n]$ is not surjective for every $n\ge n_0$.   Now let $g_{n_0}\in G[l^{n_0}]$ be an element that  is not in the image of the map $D[l^{n_0}]\to G[l^{n_0}]$.  Since $G$ is $l$-divisible (being the image of the $l$-divisible group $D$), we can lift $g_n$ to an element $(g_n)\in T_lG$.   Clearly $(g_n)$ is not the image of any element $(d_n)\in T_lD$, since then $d_{n_0}\mapsto g_{n_0}$.  Thus $T_lD\to T_lG$ is not surjective. This completes the proof of the converse.
	
	Finally assume that $D_{\operatorname{tors}}\to G_{\operatorname{tors}}$ is surjective.  Then we can simply replace the short exact sequence  $0\to H\to D\to G\to 0$ with 
	$$
	0\to H_{\operatorname{tors}}\to D_{\operatorname{tors}}\to G_{\operatorname{tors}} \to 0
	$$
	and we have reduced to the previous case, since $D$ divisible implies that $D_{\operatorname{tors}}$ is divisible, and $T_lA=T_l(A_{\operatorname{tors}})$ for any abelian group $A$.
\end{proof}

\begin{rem}\label{R:simple} Given a regular homomorphism $\phi$, using that  surjective regular homomorphisms are surjective on torsion (see \cite[Rem.~3.3]{ACMVdmij}),
	one can in fact show that $\ker(\phi)/N = \ker(\phi_{\mathrm{tors}})/N$
	for any non-zero integer $N$.
\end{rem}

\section{The Walker Abel--Jacobi map} \label{S:Walker} 
The aim of this section is to provide a new construction of the Walker Abel--Jacobi map (Theorem~\ref{T:Walker}), based on our general lifting Proposition~\ref{P:LiftReg}.

\subsection{The Bloch map and the coniveau filtration} Recall that, for any smooth projective variety $X$ over an algebraically closed field and for any prime $\ell$ invertible in $X$, Bloch~\cite{bloch} has defined a map 
$\lambda^p : \operatorname{CH}^p(X)[\ell^\infty] \to H^{2p-1}_{\mathrm{\acute et}}(X,\mathbb Q_\ell/\integ_\ell(p)).$ 
In case $X$ is a smooth projective complex variety, we obtain by comparison isomorphism a map
$\lambda^p : \operatorname{CH}^p(X)[\ell^\infty] \to H^{2p-1}_{}(X^{\mathrm{an}},\mathbb Q_\ell/\integ_\ell(p))$. 
When restricted to homologically trivial cycles, the Bloch map factors as (see, e.g., \cite[\S A.5]{ACMVbloch})
\begin{equation*}
\begin{tikzcd}
{\operatorname{CH}^p(X)_{\mathrm{hom}}[\ell^\infty]} && {H^{2p-1}(X^{\mathrm{an}},\integ(p))\otimes \rat_\ell/\integ_\ell}
\arrow["{\lambda^p}", from=1-1, to=1-3]
\arrow[from=1-3, to=1-4, hookrightarrow] &  {H^{2p-1}(X^{\mathrm{an}},\mathbb Q_\ell/\integ_\ell(p)),}
\end{tikzcd}
\end{equation*}	
where the right-hand side arrow is the canonical inclusion coming from the universal coefficient theorem.
The following lemma is due to Suzuki~\cite{Suzuki-MRL}\,:

\begin{lem}\label{L:Suzuki}
	Let $X$ be a projective complex manifold.
	Then the restriction of the Bloch map 
	$\lambda^p$ to algebraically trivial cycles factors uniquely as\,:
	\begin{equation*}
	\begin{tikzcd}
	&& {} & {\coniveau^{p-1}H^{2p-1}_{}(X^{\mathrm{an}},\integ(p))\otimes \rat_\ell/\integ_\ell} \\
	{\operatorname{A}^p(X)[\ell^\infty]} &&& {H^{2p-1}(X^{\mathrm{an}},\integ(p))\otimes \rat_\ell/\integ_\ell}
	\arrow["{\lambda^p}"', from=2-1, to=2-4]
	\arrow[from=1-4, to=2-4]
	\arrow["{\lambda_W^p}", from=2-1, to=1-4]  \arrow[from=2-4, to=2-5, hookrightarrow] &  {H^{2p-1}(X^{\mathrm{an}},\mathbb Q_\ell/\integ_\ell(p)),}
	\end{tikzcd}
	\end{equation*}	
	where the vertical arrow is induced by the inclusion $\coniveau^{p-1}H^{2p-1} (X^{\mathrm{an}},\integ(p)) \subseteq H^{2p-1} (X^{\mathrm{an}},\integ(p))$.
\end{lem}

\begin{proof} The factorization in the bottom row 
	was given above.
	The rest is obtained  in the proof of \cite[Lem.~2.2]{Suzuki-MRL} as a consequence of \cite[Thm.~5.1]{Ma}. The unicity of the factorization follows from the elementary Fact~\ref{F:el}, together with the divisibility of~$\operatorname{A}^p(X)$ (e.g., \cite[Lem.~7.10]{BlochOgus}) and the finiteness of torsion in $H^{2p-1}(X^{\mathrm{an}},\integ(p))$.
\end{proof}

\subsection{The Abel--Jacobi map on torsion and the Bloch map}\label{SS:AJ-Bloch}
Let $X$ be a projective complex manifold.
We have the  canonical identification
\begin{align}
J^{2p-1}(X)[\ell^\infty]&=  H^{2p-1}(X^{\mathrm{an}},\mathbb Z(p)) \otimes\mathbb Q_\ell/\mathbb Z_\ell  \label{E:IJtors}
\end{align}
which comes from the classical identification $J(H)[N] = H_1(J(H),  \integ/N\integ) = H_\tau \otimes \integ/N\integ$ for a pure integral Hodge structure $H$ of weight $-1$, and the elementary fact that the torsion-free quotient map $H \twoheadrightarrow H_\tau$ becomes an isomorphism after tensoring with a divisible group.
After making the identification \eqref{E:IJtors}, the Bloch map coincides with the Abel--Jacobi map on torsion. Precisely\,:
\begin{pro}[Bloch \cite{bloch}] \label{P:AJ-Bloch}
On homologically trivial cycles of $\ell$-primary torsion, the Bloch map   coincides with the Abel--Jacobi map, i.e.,  the following diagram commutes\,:
	\begin{equation*}\label{E:AJ-Bloch}
	\begin{tikzcd}
	& & J^{2p-1}(X)[\ell^\infty] \\
	{\operatorname{CH}^p(X)_{\mathrm{hom}}[\ell^\infty]} && {H^{2p-1}(X^{\mathrm{an}},\integ(p))\otimes \rat_\ell/\integ_\ell}
	\arrow["{\lambda^p}"', from=2-1, to=2-3]
	\arrow["\eqref{E:IJtors}", from=1-3, to=2-3, equal]
	\arrow["{\mathrm{AJ}[\ell^\infty]}", from=2-1, to=1-3].
	\end{tikzcd}
	\end{equation*}	
\end{pro}

\begin{proof}
	This is due to Bloch~\cite[Prop.~3.7]{bloch} (see also \cite[\S A.2.1]{ACMVbloch}).
\end{proof}

\subsection{Proof of Theorem \ref{T:Walker}} \label{SS:Walker}
Let $X$ be a projective complex manifold.
As above in \S \ref{SS:AJ-Bloch}, we have a canonical identification
\begin{align}
J^{2p-1}_{W}(X)[\ell^\infty]&= \coniveau ^{p-1}H^{2p-1}(X^{\mathrm{an}},\mathbb Z(p)) \otimes\mathbb Q_\ell/\mathbb Z_\ell.  \label{E:JWtors}
\end{align}
We are trying to construct a lift
\begin{equation*}
\xymatrix{
	& &J^{2p-1}_W(X)(\mathbb C) \ar[d]^\alpha\\
	\operatorname{A}^p(X)\ar[rr]_<>(0.5){\psi^p} \ar@{-->}[rru]^{\psi^p_W}&& J^{2p-1}_a(X)(\mathbb C) \ar@{^(->}[r]& J^{2p-1}(X)(\mathbb C).
}
\end{equation*} 
From Proposition  \ref{P:LiftReg} it suffices to construct for all primes $\ell$  a lift 
\begin{equation*}
\xymatrix{
	&& J^{2p-1}_W(X)[\ell^\infty] \ar[d]^{\alpha[\ell^\infty]}\\
	\operatorname{A}^p(X)[\ell^\infty]\ar[rr]_<>(0.5){\psi^p[\ell^\infty]} \ar@{-->}[rru]^{\psi^p_{\infty}}& &J^{2p-1}_a(X)[\ell^\infty] \ar@{^(->}[r]& J^{2p-1}(X)[\ell^\infty]. 
}
\end{equation*}
Using the identifications  \eqref{E:IJtors} and \eqref{E:JWtors}, we have a commutative diagram
\begin{equation}\label{E:BlochWalker}
\xymatrix@C=1em{
	&& J^{2p-1}_W(X)[\ell^\infty] \ar[d]^{\alpha[\ell^\infty]} \ar@{=}[rr]& &\coniveau^{p-1}H^{2p-1}_{}(X^{\mathrm{an}},\integ(p)) \otimes \rat_\ell/\integ_\ell \ar[d]\\
	\operatorname{A}^p(X)[\ell^\infty]\ar[rr]^<>(0.5){\psi^p[\ell^\infty]}  && J^{2p-1}_a(X)[\ell^\infty] \ar@{^(->}[r]& J^{2p-1}(X)[\ell^\infty] \ar@{=}[r]&    H^{2p-1}(X^{\mathrm{an}},\integ(p)) \otimes \mathbb Q_\ell/\integ_\ell  
}
\end{equation}
where, by Proposition~\ref{P:AJ-Bloch}, the composition of the bottom row is the Bloch map, and the right vertical arrow is induced by the inclusion $\coniveau^{p-1}H^{2p-1} (X^{\mathrm{an}},\integ(p)) \subseteq H^{2p-1} (X^{\mathrm{an}},\integ(p))$. The desired lift on $\ell$-power torsion is then an immediate consequence of Lemma \ref{L:Suzuki}, completing the proof of the theorem. \qed

\subsection{The Walker Abel--Jacobi map does not lift further along isogenies}
The following result was communicated to us by Fumiaki Suzuki.
\begin{teo}[Suzuki]\label{T:Walker-initial}
	Suppose $X$ is a projective complex manifold. Then the kernel of the Walker Abel--Jacobi map $\psi_W^p : \operatorname{A}^p(X) \to J^{2p-1}_W(X)$ is divisible. Consequently, 
	the Walker Abel--Jacobi map $\psi_W^p$ is initial among all lifts of the Abel--Jacobi map $\psi^p : \operatorname{A}^p(X) \to J_a^{2p-1}(X)$ along isogenies\,; in particular, if $\psi^p_W: \operatorname{A}^p(X) \to J^{2p-1}_W(X)$ factors through an isogeny $f : A \to J^{2p-1}_W(X)$, then $f$ is an isomorphism.
\end{teo}

\begin{proof} 
By Theorem~\ref{T:factorization} and Corollary~\ref{C:factorization},
 it is equivalent to show that the kernel of the restriction of $\psi^p_W$ to $\ell$-primary torsion is divisible for all primes $\ell$.
 By the short exact sequence of \cite[Lem.~2.2]{Suzuki-MRL}, $\ker (\psi^p_W[\ell^\infty])$ is a quotient of $K\otimes \rat_\ell/\integ_\ell$, where $K$ is the kernel of the surjection $$f^p : H^{p-1}(X,\mathcal{H}^p(\integ(p))) \to \coniveau^{p-1}H^{2p-1}(X,\integ(p)).$$
The divisibility of $\ker (\psi^p_W[\ell^\infty])$ then follows from the divisibility of $K\otimes \rat_\ell/\integ_\ell$. (For any abelian group $A$ we have $A\otimes \mathbb Q_\ell/\mathbb Z_\ell$ is divisible.)
\end{proof}

\begin{rem}
For a complex projective  manifold $X$, the kernel of the Abel--Jacobi map $\psi^p : \operatorname{A}^p(X) \to J_a^{2p-1}(X)$ is not divisible in general for $p>2$, as shown by the construction of \cite[Cor.~4.2]{Ottem-Suzuki}.
\end{rem}

\subsection{First proof of Theorem~\ref{T:main}} \label{SS:firstproof}
Recall from~\cite[Thm.~A]{ACMVdmij} that, given  a smooth projective variety $X$ over a subfield $K$ of $\cx$, the algebraic intermediate Jacobian $J_a^{2p-1}(X_\cx)$ admits a unique model over $K$ such that the Abel--Jacobi map $\psi^p : \operatorname{A}^p(X_{\cx}) \to J_a^{2p-1}(X_\cx)$
is $\operatorname{Aut}(\cx/K)$-equivariant. 
By Theorem~\ref{T:Walker-initial}, the Walker Abel--Jacobi map $\psi_W^p : \operatorname{A}^p(X_\cx) \to J^{2p-1}_W(X_\cx)$ is universal among lifts of the Abel--Jacobi map along isogenies. 
We can conclude from Theorem~\ref{T:factorization} that the Walker intermediate Jacobian $J^{2p-1}_W(X_\cx)$ admits a unique model over $K$ such that the Walker Abel--Jacobi map 
$\psi_W^p : \operatorname{A}^p(X_{\cx}) \to J_W^{2p-1}(X_\cx)$
is $\operatorname{Aut}(\cx/K)$-equivariant.
\qed

\section{Descending the Walker Abel--Jacobi map}

In this section we provide a second proof of Theorem~\ref{T:main}. It is based on a factorization of the Bloch map restricted to algebraically trivial cycles. This approach will prove crucial for our applications (Corollaries~\ref{C:model} and~\ref{C:imageBloch}).
For that purpose,
 we start by recasting the results of \S \ref{S:Walker} in the $\ell$-adic setting.

\subsection{The Bloch map and the coniveau filtration, $\ell$-adically}

For lack of a suitable reference, we start with a comparison between the analytic and $\ell$-adic coniveau filtrations\,:

\begin{lem}\label{L:ConiGAGAbody}
	Let $X$ be a smooth projective variety over a field $K\subseteq \mathbb C$.  
	We have canonical identifications
	\begin{equation}\label{E:identificationBody}
	{\coniveau^iH_{\operatorname{\acute{e}t}}^j(X_{\mathbb C},\integ_\ell)} =  {\coniveau^iH^j(X_{\mathbb C}^{\mathrm{an}},\integ)\otimes \integ_\ell}.
	\end{equation}
Moreover, the natural action of $\operatorname{Aut}(\mathbb C/K)$-action on
 $H_{\operatorname{\acute{e}t}}^j(X_{\mathbb C},\integ_\ell)$ induces an action on 
	${\coniveau^iH_{\operatorname{\acute{e}t}}^j(X_{\mathbb C},\integ_\ell)} $.
\end{lem}
\begin{proof} 
	We have the following commutative diagram\,:
	\[\begin{tikzcd}
	{0} & {\coniveau^iH_{\operatorname{\acute{e}t}}^j(X_{\mathbb C},\integ_\ell)} & {H_{\operatorname{\acute{e}t}}^j(X_{\mathbb C},\integ_\ell)} & {\varinjlim H_{\operatorname{\acute{e}t}}^j(X_{\mathbb C}\setminus Z,\integ_\ell)} \\
	{0} & {\coniveau^iH^j(X_{\mathbb C}^{\mathrm{an}},\integ_\ell)} & {H^j(X_{\mathbb C}^{\mathrm{an}},\integ_\ell)} & {\varinjlim H^j((X_{\mathbb C}\setminus Z)^{\mathrm{an}},\integ_\ell)} \\
	{0} & {\coniveau^iH^j(X_{\mathbb C}^{\mathrm{an}},\integ)\otimes \integ_\ell} & {H^j(X_{\mathbb C}^{\mathrm{an}},\integ)\otimes \integ_\ell} & {\varinjlim H^j((X_{\mathbb C}\setminus Z)^{\mathrm{an}},\integ)\otimes \integ_\ell}
	\arrow[from=1-1, to=1-2]
	\arrow[from=2-1, to=2-2]
	\arrow[from=3-1, to=3-2]
	\arrow[from=1-2, to=1-3]
	\arrow[from=2-2, to=2-3]
	\arrow[from=3-2, to=3-3]
	\arrow[from=1-3, to=1-4]
	\arrow[from=2-3, to=2-4]
	\arrow[from=3-3, to=3-4]
	\arrow["{\simeq}", from=3-3, to=2-3]
	\arrow["{\simeq}", from=3-4, to=2-4]
	\arrow["{\simeq}", from=2-3, to=1-3]
	\arrow["{\simeq}", from=2-4, to=1-4]
	\end{tikzcd}\]
	Here the limits are taken over all closed subschemes $Z$ of $X_\cx$ of codimension $\leq i$.
	The top two rows are exact by definition of the coniveau filtration, while the third is also exact by flatness of the $\integ$-module $\integ_\ell$. The bottom vertical arrows are isomorphisms by flatness of $\integ_\ell$ and the fact that $\varinjlim$ commutes with $\otimes$. 
	The top two vertical arrows are the isomorphisms provided by Artin's comparison theorem. 
	Thus we obtain the desired identification.
	
	The action of $\operatorname{Aut}(\mathbb C/K)$ on $
	{\coniveau^iH_{\operatorname{\acute{e}t}}^j(X_{\mathbb C},\integ_\ell)} $ comes from the fact that the coniveau filtration on~$X_{\mathbb C}$ can be obtained using subvarieties defined over $K$ (as can be seen by spreading out and by using smooth base-change, followed by taking Galois-orbits).  
\end{proof}

As an immediate consequence of Lemma~\ref{L:ConiGAGAbody}, we obtain\,:

\begin{lem}\label{L:suzuki2}
	Let $X$ be a smooth projective variety over a field $K\subseteq \mathbb C$.
	Then the restriction of the $\operatorname{Aut}(\mathbb C/K)$-equivariant  Bloch map $\lambda^p : \operatorname{CH}^p(X_{\mathbb C})[\ell^\infty] \to H^{2p-1}_{\operatorname{\acute{e}t}}(X_{\mathbb C},\mathbb Q_\ell/\integ_\ell(p))$ to algebraically trivial cycles factors uniquely  into the following commutative diagram of $\operatorname{Aut}(\mathbb C/K)$-modules\,:
	$$
	\xymatrix{
		&&{\coniveau^{p-1}H^{2p-1}_{\operatorname{\acute{e}t}}(X_{\mathbb C},\integ_\ell(p))\otimes_{\mathbb Z_\ell} \rat_\ell/\integ_\ell}\ar[d]&\\
		{\operatorname{A}^p(X_{\mathbb C})[\ell^\infty]} \ar[rr]_<>(0.5){\lambda^p} \ar[rru]^<>(0.5){\lambda^p_W\quad }&& {H^{2p-1}_{\operatorname{\acute{e}t}}(X_{\mathbb C},\integ_\ell(p))\otimes_{\mathbb Z_\ell} \rat_\ell/\integ_\ell}
		\ar[r] & H^{2p-1}_{\operatorname{\acute{e}t}}(X_{\mathbb C},\mathbb Q_\ell/\integ_\ell(p)),
	}
	$$
	where the vertical arrow is induced by the inclusion $\coniveau^{p-1}H^{2p-1}_{\operatorname{\acute{e}t}}(X_{\mathbb C},\integ_\ell(p)) \subseteq H^{2p-1}_{\operatorname{\acute{e}t}}(X_{\mathbb C},\integ_\ell(p))$.  
\end{lem}

\begin{proof}
	The factorization as groups follows directly from Lemma~\ref{L:Suzuki} together with the identification~\eqref{E:identificationBody}.  
	Now,  since $H^{2p-1}_{\operatorname{\acute{e}t}}(X_{\mathbb C},\integ_\ell(p))$ has finite rank and finite torsion, the elementary Fact~\ref{F:el} shows that the lift $\lambda^p_W$ is uniquely determined by $\lambda^p$. In addition, still by Fact~\ref{F:el}, since both $\lambda^p$ and the inclusion $\coniveau^{p-1}H^{2p-1}_{\operatorname{\acute{e}t}}(X_{\mathbb C},\integ_\ell(p)) \subseteq H^{2p-1}_{\operatorname{\acute{e}t}}(X_{\mathbb C},\integ_\ell(p))$ are $\operatorname{Aut}(\cx/K)$-equivariant, then so is $\lambda^p_W$.
\end{proof}

\subsection{The Walker Abel--Jacobi map on torsion and the Bloch map, $\ell$-adically}
From the identification \eqref{E:JWtors} and Lemma~\ref{L:ConiGAGAbody}, we obtain the following canonical identification of abelian groups\,: 
\begin{align}
J^{2p-1}_{W}(X_{\mathbb C})[\ell^\infty]&
= \coniveau^{p-1}H^{2p-1}_{\operatorname{\acute{e}t}}(X_{\mathbb C},\mathbb Z_\ell(p)) \otimes_{\mathbb Z_\ell}\mathbb Q_\ell/\mathbb Z_\ell, \label{E:JWtorsEt}
\end{align}
which is the $\ell$-adic analogue of the identification \eqref{E:JWtors}.
In addition, by the comparison isomorphism in cohomology, Proposition~\ref{P:AJ-Bloch} provides a commutative diagram\,:
\begin{equation*}\label{E:AJ-Bloch-ell}
\begin{tikzcd}
& & J^{2p-1}(X)[\ell^\infty] \\
{\operatorname{CH}^p(X)_{\mathrm{hom}}[\ell^\infty]} && {H^{2p-1}_{\operatorname{\acute{e}t}}(X_{\mathbb C},\integ_\ell(p))\otimes_{\mathbb Z_\ell} \rat_\ell/\integ_\ell.}
\arrow["{\lambda^p}"', from=2-1, to=2-3]
\arrow[from=1-3, to=2-3, equal]
\arrow["{\mathrm{AJ}[\ell^\infty]}", from=2-1, to=1-3] 
\end{tikzcd}
\end{equation*}	
The following lemma will play a crucial role in the proof of Theorem~\ref{T:main}. It shows that, \emph{via} the identification \eqref{E:JWtorsEt}, the restriction of the Walker Abel--Jacobi map to $\ell$-primary torsion coincides with the factorization of the Bloch map given in Lemma~\ref{L:suzuki2}.

\begin{lem}\label{L:BlochWalker}
	Let $X$ be a smooth projective variety over a field $K\subseteq \mathbb C$.  
On algebraically trivial cycles of $\ell$-primary torsion,	the map $\lambda_W^p$ coincides with the Walker Abel--Jacobi map $\psi_W^p$, i.e., the following diagram commutes\,:
	\[\begin{tikzcd}
	&& {J^{2p-1}_{W}(X_\cx)[\ell^\infty]} \\
	{\operatorname{A}^p(X_\cx)[\ell^\infty]} && {\coniveau ^{p-1}H^{2p-1}_{\operatorname{\acute et}}(X_\cx,\mathbb Z_\ell(p)) \otimes_{\mathbb Z_\ell}\mathbb Q_\ell/\mathbb Z_\ell.} 
	\arrow["{\lambda^p_W}"', from=2-1, to=2-3]
	\arrow["{\psi^p_W[\ell^\infty]}", from=2-1, to=1-3]
	\arrow["{\eqref{E:JWtorsEt}}", from=1-3, to=2-3, equal]
	\end{tikzcd}\]
\end{lem}
\begin{proof}
	This follows directly from restricting the previous diagram 
	to algebraically trivial cycles and from the fact that $\lambda_W^p$ (resp.~${\psi^p_W[\ell^\infty]}$) are the unique lifts of $\lambda^p$ (resp.~${\psi^p[\ell^\infty]}$).
\end{proof}

\subsection{Second proof of Theorem \ref{T:main}} \label{SS:main}
	Let  $X$ be a smooth projective variety over a field $K\subseteq \mathbb C$.
	Recall that we showed in \cite[Thm.~A]{ACMVdmij} (see also~\cite[Thm.~9.1]{ACMVfunctor}) that  $J^{2p-1}_a(X_\cx)$ admits a unique model over $K$ such that the Abel--Jacobi map $\psi^p : \operatorname{A}^p(X_\cx) \to J^{2p-1}_a(X_\cx)$ is $\operatorname{Aut}(\cx/K)$-equivariant. We are going to show that $\alpha$ descends uniquely to $K$ with respect to the above $K$-structure on $J^{2p-1}_a(X_\cx)$. The $\operatorname{Aut}(\cx/K)$-equivariance of  $\psi^p_W : \operatorname{A}^p(X) \to J^{2p-1}_W(X)$ then follows from the unicity of $\psi^p_W$.
	
	To that end, 
let  $C$ be a $K$-pointed, geometrically integral, smooth projective curve over $K$, together with a correspondence $\Gamma \in \operatorname{CH}^p(C\times_K X)$ such that the induced homomorphism $J(C_\cx) \to J^{2p-1}_a(X_{\mathbb C})$ is surjective. The existence of such a $C$ and $\Gamma$ is provided by \cite[Prop.~1.1]{ACMVdmij}. We thus obtain
a commutative diagram
\begin{equation}\label{E:Pfmain}
\begin{tikzcd}
{\operatorname{A}^1(C_\cx)} & {J(C_\cx)}(\mathbb C) \\
{\operatorname{A}^p(X_\cx)} & {J_W^{2p-1}(X_{\mathbb C})}(\mathbb C) & {J^{2p-1}_a(X_{\mathbb C})(\mathbb C),}
\arrow["{\psi^1=\psi^1_W}", from=1-1, to=1-2]
\arrow["{\Gamma_*}"', from=1-1, to=2-1]
\arrow["{\psi^p_W}", from=2-1, to=2-2]
\arrow["{\alpha}", from=2-2, to=2-3, two heads]
\arrow["{\gamma}"', from=1-2, to=2-2, two heads]
\arrow[from=1-2, to=2-3, two heads]
\end{tikzcd}
\end{equation}
where the homomorphism $\gamma$, which is defined by the fact that the Jacobian of a curve together  with the Abel map is  a universal regular homomorphism,  is also induced by  the correspondence $\Gamma_* : H^1(C_\cx^{\mathrm{an}},\integ(1)) \to H^{2p-1}(X_{\mathbb C}^{\mathrm{an}},\integ(p))$ (which factors through $\coniveau^{p-1}H^{2p-1}(X_{\mathbb C}^{\mathrm{an}},\integ(p))$\,; see e.g.,~\cite[Prop.~1.1]{ACMVbloch}).
We then  show that, with respect to the $K$-structure on $J(C_\cx)$ given by the Jacobian $J(C)$ of~$C$, the surjective homomorphism $\gamma$ descends to $K$. (That $\alpha \circ \gamma$ descends to~$K$ was established in \cite[\S2]{ACMVdmij}.) For that purpose, by the elementary \cite[Lem.~2.3]{ACMVdmij}, it suffices to show that, for all primes $\ell$, the $\ell$-primary torsion in 
$$P:= \ker \big(\gamma: J(C_\cx) \to J_W^{2p-1}(X_\cx)\big)$$ 
is stable under the action of $\operatorname{Aut}(\cx/K)$ on $J(C)(\cx)$.

For this we take $\ell$-primary torsion in the commutative diagram \eqref{E:Pfmain}, then use the compatibility of the Bloch map with the Walker Abel--Jacobi map (Lemma~\ref{L:BlochWalker}) to obtain the commutative diagram
$$
\begin{small}
\xymatrix@C=1em@R=1.5em{
&&&0 \ar[d]\\
&&&P[\ell^\infty] \ar[d]\\
\operatorname{A}^1(C_{\mathbb C})[\ell^\infty]\ar[rr]^<>(0.5){\lambda^1[\ell^\infty]}_<>(0.5){\sim} \ar[d]^{\Gamma_*}
\ar@/^3pc/[rrr]^{\psi^1[\ell^\infty]}
	&& H_{\operatorname{\acute{e}t}}^1(C_\cx,\integ_\ell(1))\otimes \rat_\ell/\integ_\ell  \ar[d]^{\Gamma_*}  \ar[d] \ar@{=}[r]&J(C_{\mathbb C})[\ell^\infty] \ar[d]^\gamma\\
\operatorname{A}^p(X_{\mathbb C})[\ell^\infty]\ar[rr]^<>(0.5){\lambda^p_W[\ell^\infty]}
\ar@/_2.5pc/[rrr]^{\psi^p_W[\ell^\infty]} &&\coniveau^{p-1}H_{\operatorname{\acute{e}t}}^{2p-1}(X_{\mathbb C},\integ_\ell(p))\otimes \rat_\ell/\integ_\ell  \ar@{=}[r]& J_W^{2p-1}(X_{\mathbb C})[\ell^\infty]\ar[d]\\
&&&0
}
\end{small}
$$
The only things
 that needs explaining is the middle vertical map\,: here we are using the fact that the Bloch map is compatible with correspondences, and the fact mentioned above  that the correspondence $\Gamma_* : H^1_{\operatorname{\acute{e}t}}(C_\cx,\integ_\ell (1)) \to H_{\operatorname{\acute{e}t}}^{2p-1}(X_{\mathbb C},\integ_\ell(p))$ factors through $\coniveau^{p-1}H_{\operatorname{\acute{e}t}}^{2p-1}(X_{\mathbb C},\integ_\ell(p))$.  (Although we do not strictly  need it for the argument, we note that  the left hand square is $\operatorname{Aut}(\mathbb C/K)$-equivariant due to   Lemma \ref{L:ConiGAGAbody}.)

Therefore $P[\ell^\infty]$ is identified with the kernel of 
$$\Gamma_* : {H_{\operatorname{\acute{e}t}}^1(C_\cx,\integ_\ell(1))\otimes \rat_\ell/\integ_\ell} \to {\coniveau^{p-1}H_{\operatorname{\acute{e}t}}^{2p-1}(X_\cx,\integ_\ell(p))\otimes \rat_\ell/\integ_\ell},
$$ which, since $\Gamma$ is defined over $K$, is stable under the action of $\operatorname{Aut}(\cx/K)$.
We have thus showed that $\gamma$ descends to $K$. 
Combined with the fact \cite[\S 2]{ACMVdmij} that $\alpha \circ \gamma$ also descends to $K$ with respect to the $K$-structure of $J(C_\cx)$ given by $J(C_\cx) = J(C)_\cx$, we readily obtain that $\alpha$ descends to~$K$ (e.g., by the elementary \cite[Lem.~2.4]{ACMVdmij}).
\qed

\subsection{Further remarks}
\begin{rem}[Base change of field]\label{R:L/K}
If $X$ is a smooth projective variety over a field  $K\subseteq L\subseteq \mathbb C$, then there is a canonical identification $J^{2p-1}_{W,X_L/L}=(J^{2p-1}_{W,X/K})_L$. 
\end{rem}

\begin{rem}[Independence of embedding of $K$ in $\mathbb C$] \label{R:EmbKinCC}
	Let $X$ be a smooth \emph{complex} projective  variety.
	  For a smooth projective complex variety $Z$ and an automorphism $\sigma \in \operatorname{Aut}(\cx)$, we denote $Z^\sigma := Z \otimes_\sigma \cx$ the base-change of $Z$ along $\sigma$.
Arguing as in the proof of \cite[Prop.~3.1]{ACMVnormal} shows the following extension of Theorem~\ref{T:main}\,:  for all $\sigma \in \operatorname{Aut}(\cx)$ there is a canonical identification
$$J^{2p-1}_W(X^\sigma) = J^{2p-1}_W(X)^\sigma$$ and a commutative diagram
\[\begin{tikzcd}
{\operatorname{A}^p(X)} & {J^{2p-1}_W(X)} \\
{\operatorname{A}^p(X^\sigma)} & {J^{2p-1}_W(X^\sigma).}
\arrow["{\sigma^*}"', from=1-1, to=2-1]
\arrow["{\sigma^*}", from=1-2, to=2-2]
\arrow["{\psi^p_{W,X}}", from=1-1, to=1-2]
\arrow["{\psi^p_{W,X^\sigma}}", from=2-1, to=2-2]
\end{tikzcd}\]
As a consequence, for a smooth projective variety $X$ over a field $K$ of characteristic $0$, 
the kernel of the Walker Abel--Jacobi map associated to $X$ and an embedding of $K$ into $\mathbb C$ is independent of that embedding.
\end{rem}

\section{Applications regarding the coniveau filtration}

\subsection{Modeling coniveau -- on a question of Mazur} \label{SS:Mazur}
In this paragraph, we show Corollary~\ref{C:model} stating that the model of the Walker intermediate Jacobian over $K$ from Theorem~\ref{T:main} models the torsion-free quotient of 
$\coniveau^{p-1}H_{\operatorname{\acute{e}t}}^{2p-1}(X_{\bar K},\integ_\ell)$.

\begin{proof}[Proof of Corollary \ref{C:model}]
Let $J^{2p-1}_{W,X/K}$ be the model over $K$, provided by Theorem~\ref{T:main}, of the Walker intermediate Jacobian $J^{2p-1}_W(X_{\mathbb C})$ making $\psi^p_W$ an $\operatorname{Aut}(\cx/K)$-equivariant homomorphism.
By the very construction of $J^{2p-1}_{W,X/K}$, and the identification~\eqref{E:JWtorsEt}, we have for all primes $\ell$ an $\operatorname{Aut}(\cx/K)$-equivariant identification 
$T_\ell J^{2p-1}_{W,X/K} = \coniveau^{p-1}H_{\operatorname{\acute{e}t}}^{2p-1}(X_{\mathbb C},\integ_\ell)_\tau$, 
thereby concluding the proof of Corollary~\ref{C:model}.
\end{proof}

\begin{rem}
	Following on Remark~\ref{R:EmbKinCC}, for $X$ a smooth complex projective  variety, the identifications of Corollary~\ref{C:model} more generally fit in the commutative diagram
	\[\begin{tikzcd}
	{T_\ell J^{2p-1}_{W}(X)} & {\coniveau^{p-1}H_{\operatorname{\acute{e}t}}^{2p-1}(X,\integ_\ell)_\tau} \\
	{T_\ell J^{2p-1}_{W}(X^\sigma)} & {\coniveau^{p-1}H_{\operatorname{\acute{e}t}}^{2p-1}(X^\sigma,\integ_\ell)_\tau}
	\arrow["{\sigma^*}"', from=1-1, to=2-1]
	\arrow["{\sigma^*}", from=1-2, to=2-2]
	\arrow["{}", from=1-1, to=1-2, equal]
	\arrow["{}", from=2-1, to=2-2, equal]
	\end{tikzcd}\]
	for every $\sigma \in \operatorname{Aut}(\cx)$.
\end{rem}

\subsection{The image of the $\ell$-adic Bloch map}
Following up on \cite{ACMVbloch}, Corollary~\ref{C:imageBloch} determines exactly the image of the $\ell$-adic Bloch map in case the base field $K$ has zero characteristic. Here is a proof of that corollary\,:

\begin{proof}[Proof of Corollary~\ref{C:imageBloch}]
By the Lefschetz principle we may and do assume $K\subseteq \cx$. By rigidity, it suffices to establish the proposition after base-change to $\cx$.
Taking Tate modules in the commutative diagram of Lemma~\ref{L:BlochWalker}, we obtain a commutative diagram
	\[\begin{tikzcd}
&& {T_\ell J^{2p-1}_{W}(X_\cx)} \\
{T_\ell \operatorname{A}^p(X_\cx)} && {\coniveau ^{p-1}H^{2p-1}_{\operatorname{\acute et}}(X_\cx,\mathbb Z_\ell(p))_\tau .} 
\arrow["{T_\ell\lambda^p_W}"', from=2-1, to=2-3]
\arrow["{T_\ell\psi^p_W}", from=2-1, to=1-3]
\arrow[from=1-3, to=2-3, equal]
\end{tikzcd}\]
The proposition then follows from the fact that the Walker Abel--Jacobi map does not lift along non-trivial isogenies (Theorem~\ref{T:Walker-initial}) and from the equivalence of (3) and (4) in Corollary~\ref{C:factorization}.
\end{proof}

\bibliographystyle{amsalpha}
 \bibliography{DCG}

\end{document}